\pdfoutput = 1
\newif\iffinal
\finaltrue
\documentclass[a4paper]{amsart}
\usepackage{amssymb}
\usepackage{amsthm}
\usepackage{mathtools}
\usepackage{hyperref}
\usepackage[utf8]{inputenc}
\usepackage[T1]{fontenc}

\iffinal
\newcommand{\TODO}[1]{}
\else
\usepackage[mark]{gitinfo2}
\renewcommand{\gitMark}{\jobname\,\textbullet{}\,\gitFirstTagDescribe\,\textbullet{}\,\gitAuthorName,\,\gitAuthorIsoDate}
\newcommand{\TODO}[1]
{\par\fbox{\begin{minipage}{0.9\linewidth}\textbf{TODO:} #1\end{minipage}}\par}
\fi

\DeclarePairedDelimiter{\abs}{\lvert}{\rvert}

\newcommand{\C}{\mathbb{C}}
\newcommand{\calF}{\mathcal{F}}
\newcommand{\calM}{\mathcal{M}}
\newcommand{\DLMF}[2]{\cite[\href{http://dlmf.nist.gov/#1.E#2}{#1.#2}]{NIST:DLMF:v1.0.10}}
\newcommand{\E}{\mathbb{E}}

\DeclareMathOperator{\Li}{Li}
\renewcommand{\MR}[1]{}
\newcommand{\phibar}{\overline\phi}
\renewcommand{\P}{\mathbb{P}}
\DeclareMathOperator{\Res}{Res}

\newcommand{\Z}{\mathbb{Z}}

\newtheorem*{theoremnn}{Theorem}
\newtheorem{theorem}{Theorem}

\newtheorem{lemma}{Lemma}
\theoremstyle{remark}
\newtheorem*{remark}{Remark}

\author{Clemens Heuberger}
\address[Clemens Heuberger]{Institut f\"ur Mathematik, Alpen-Adria-Uni\-ver\-si\-t\"at Klagenfurt,
  Universit\"atsstra\ss e 65--67, 9020 Klagenfurt, Austria}
\email{clemens.heuberger@aau.at}
\thanks{C.~Heuberger is supported by the Austrian Science Fund (FWF):
  P~24644-N26. Parts of this paper have been written while he was a visitor at Stellenbosch University.}

\author{Stephan Wagner}
\address[Stephan Wagner]{Department of Mathematical Sciences, Stellenbosch University, 7602 Stellenbosch,
 South Africa}
\email{swagner@sun.ac.za}
\thanks{S.~Wagner is supported by the National Research Foundation of South Africa, grant number 96236.}

\title{On the monoid generated by a Lucas sequence}


\begin{document}

\begin{abstract}
A Lucas sequence is a sequence of the general form $v_n = (\phi^n - \phibar^n)/(\phi-\phibar)$, where $\phi$ and $\phibar$ are real algebraic integers such that $\phi+\phibar$ and $\phi\phibar$ are both rational. Famous examples include the Fibonacci numbers, the Pell numbers, and the Mersenne numbers. We study the monoid that is generated by such a sequence; as it turns out, it is almost freely generated. We provide an asymptotic formula for the number of positive integers $\leq x$ in this monoid, and also prove Erd\H{o}s-Kac type theorems for the distribution of the number of factors, with and without multiplicity. While the limiting distribution is Gaussian if only distinct factors are counted, this is no longer the case when multiplicities are taken into account.
\end{abstract}
\maketitle

\section{Introduction}
Let $\phi$ and $\phibar$ be real algebraic integers such that $\phi+\phibar$ and
$\phi\phibar$ are fixed non-zero coprime rational integers with $\phi>\abs{\phibar}$.
The \emph{Lucas numbers} associated with $(\phi, \phibar)$ are 
\begin{equation*}
  v_n=v_n(\phi,\phibar)=\frac{\phi^n-\phibar^n}{\phi-\phibar}, \qquad
  n=1, 2, 3, \ldots.
\end{equation*}
All these numbers are positive integers, and the sequence is strictly increasing for $n \geq 2$. Famous examples include the Fibonacci numbers ($\phi = \frac{1+\sqrt{5}}{2}$, $\phibar = \frac{1-\sqrt{5}}{2}$), the Pell numbers ($\phi = 1+\sqrt{2}$, $\phibar = 1-\sqrt{2}$), and the Mersenne numbers ($\phi=2$, $\phibar=1$). The multiplicative group generated by such a sequence was studied in a recent paper by Luca et al.~\cite{Luca-Pomerance-Wagner:2011:fibon}; by definition, it consists of all quotients of products of elements of the sequence. In~\cite{Luca-Pomerance-Wagner:2011:fibon}, a near-asymptotic formula for the number of integers in this group was given (in the specific case of the Fibonacci sequence, but the method is more general). This continued earlier work of Luca and Porubsk\'y \cite{Luca-Porubsky:2003:multiplicative}, who considered the multiplicative group generated by a Lehmer sequence and showed that the number of integers below $x$ in the resulting group is $O(x/(\log x)^{\delta})$ for any positive number $\delta$.

As it turns out, the group that is generated by a Lucas sequence is almost a free group; this is due to the existence of \emph{primitive divisors}.
A prime number $p$ is a primitive divisor of $v_n(\phi,\phibar)$ if $p$ divides $v_n$ but does not divide $v_1\ldots v_{n-1}$. It is a classical result, due to Carmichael, that almost all elements of a Lucas sequence have primitive divisors:

\begin{theoremnn}[{Carmichael~\cite[Theorem~XXIII]{Carmichael:1913:numer-factor}}]
  If $n\notin\{1, 2, 6, 12\}$, then $v_n(\phi,\phibar)$ has a primitive divisor.
\end{theoremnn}
See also Bilu, Hanrot and
Voutier~\cite{Bilu-Hanrot-Voutier:2001:exist-lucas-lehmer}. Note the slightly
different definition in \cite{Bilu-Hanrot-Voutier:2001:exist-lucas-lehmer}
which does not allow a primitive divisor to divide $(\phi-\phibar)^2$.

Let now
\begin{equation*}
  \calF=\{v_n(\phi, \phibar) \mid v_n(\phi, \phibar)\text{ has a primitive divisor}\}.
\end{equation*}
By Carmichael's theorem, $\calF$ includes all but finitely many $v_n$. We let $\calF_0$ be the set of all $v_n(\phi,\phibar)$ with $n \leq 12$ that have a primitive divisor, so that
\begin{equation*}
  \calF=\calF_0 \cup \{v_n(\phi,\phibar)
  \mid n\ge 13\}.
\end{equation*}
For example, in the case of the golden ratio $\phi=(1+\sqrt{5})/2$,
$\phibar=(1-\sqrt{5})/2$, we have
\begin{equation*}
  \calF_0=\{2, 3, 5, 13, 21, 34, 55, 89\}
\end{equation*}
and
\begin{equation*}
  \calF=\{2, 3, 5, 13, 21, 34, 55, 89, 233, 377, \ldots\},
\end{equation*}
which are all the Fibonacci numbers except $1$, $8$ and $144$.

Instead of the full group that was studied in~\cite{Luca-Porubsky:2003:multiplicative} and \cite{Luca-Pomerance-Wagner:2011:fibon}, we consider the free monoid
\begin{equation*}
\calM(\calF)=\{m_1\ldots m_k \mid k\ge 0, m_j\in\calF\}
\end{equation*}
generated by $\calF$. It is an easy consequence of the existence of primitive
divisors that every element of $\calM(\calF)$ has a \emph{unique} factorisation into elements of
$\calF$: If an element has two factorisations, consider a primitive divisor of
the largest element of $\calF$ occurring in either of those factorisations. This
prime number has to occur in all factorisations, so the largest element of
$\calF$ occurring in any factorisation of the element is fixed. The result
follows by induction.

To give a concrete example, the elements of our monoid are
\begin{equation*}
\calM(\calF)=\{1,2,3,4,5,6,8,9,10,12,13,15,16,18,20,21,24,25,26,27,30,32,34,\ldots\}
\end{equation*}
in the case of the Fibonacci numbers (cf. \cite[A065108]{OEIS:2016}).

Our first result describes the asymptotic number of elements in $\calM(\calF)$ up
to a given bound, paralleling the aforementioned result of \cite{Luca-Pomerance-Wagner:2011:fibon}, but even being more precise. All constants, implicit constants in $O$-terms and the
Vinogradov notation will depend on $\phi$ and $\phibar$.

\begin{theorem}\label{theorem:count}
  We have
  \begin{equation*}
    \abs{\calM(\calF) \cap [1, x]} = k_0 (\log x)^{k_1}
    \exp\biggl(\pi\sqrt{\frac{2\log x}{3\log\phi}}\biggr)\Bigl(1+O\Bigl(\frac1{(\log x)^{1/10}}\Bigr)\Bigr)
  \end{equation*}
  for $x\to \infty$ and suitable constants $k_0$ and $k_1$. Specifically,
\begin{equation*}
k_1 = \frac{\abs{\calF_0} - 13}{2} + \frac{\log(\phi - \phibar)}{2\log \phi}.
\end{equation*}
\end{theorem}
\begin{remark}
An explicit expression for $k_0$ can be given as well, but it is rather unwieldy.
\end{remark}

Since all elements of $\calM(\calF)$ have a unique factorisation into elements of $\calF$, it makes sense to consider the number of factors in this factorisation and to study its distribution. The celebrated Erd\H{o}s-Kac theorem \cite{Erdoes-Kac:1940:gauss} states that the number of distinct prime factors of a randomly chosen integer in $[1,x]$ is asymptotically normally distributed. The same is true if primes are counted with multiplicity. We refer to Chapter 12 of \cite{Elliott:1980:probabilistic} for a detailed discussion of the Erd\H{o}s-Kac theorem and its generalisations.

Our aim is to prove similar statements for the monoid $\calM(\calF)$. Let $n$ be an element of this monoid. By $\omega_\calF(n)$ and $\Omega_\calF(n)$, we denote the number of factors in the
factorisation of $n$ into elements of $\calF$ without and with multiplicities, respectively.

We first prove asymptotic normality for $\omega_\calF$, in complete analogy with the Erd\H{o}s-Kac theorem:

\begin{theorem}\label{theorem:asymptotic-normality}
  Let $N$ be a uniformly random positive integer in $\calM(\calF) \cap [1, x]$
  and let
  \begin{equation*}
    a_1 = \frac1{\pi}\sqrt{\frac{6}{\log\phi}},\qquad
    a_2 = \frac{\pi^2-6}{2\pi^3}\sqrt{\frac{6}{\log\phi}}.
  \end{equation*}
  The random variable $\omega_\calF(N)$ is asymptotically normal: we have
  \begin{equation*}
    \lim_{x\to\infty}\P\Bigl(\frac{\omega_\calF(N)-a_1\log^{1/2} x}{\sqrt{a_2}\log^{1/4}x}\le z\Bigr)=\frac1{\sqrt{2\pi}}\int_{0}^z
    e^{-y^2/2}\, dy. 
  \end{equation*}
\end{theorem}

However, the situation changes when multiplicities are taken into account: unlike the arithmetic function $\Omega$, which counts all prime factors with multiplicity, $\Omega_\calF$ is not normally distributed. Its limiting distribution can rather be described as a sum of shifted exponential random variables, as the following theorem shows:

\begin{theorem}\label{theorem:distribution-Omega}
  Let $N$ be a uniformly random positive integer in $\calM(\calF) \cap [1, x]$,
  let $a_1$ be the same constant as in the previous theorem and, with $\gamma$ denoting the Euler-Mascheroni constant,
  \begin{align*}
    b_1 &= \frac{\sqrt{6 \log \phi}}{\pi} \bigg( \frac{2\gamma - \log(\pi^2\log\phi/6)}{2\log \phi} + \sum_{m \in \calF_0} \frac{1}{\log m} + \frac{1}{\log v_{13}(\phi,\phibar)} \\
&\qquad+ \sum_{k \geq 1} \Big( \frac{1}{\log v_{k+13}(\phi,\phibar)} - \frac{1}{k \log \phi} \Big) \bigg), \\
    b_2 &= \frac{\sqrt{6 \log \phi}}{\pi}.
  \end{align*}

  The random variable $\Omega_\calF(N)$, suitably normalised, converges weakly to a sum of shifted exponentially distributed random variables:
  \begin{equation*}
\frac{\Omega_\calF(N)-\frac{a_1}{2} \log^{1/2} x \log \log x - b_1 \log^{1/2} x}{b_2 \log^{1/2} x} \overset{(d)}{\to} \sum_{m \in \calF} \Big( X_m - \frac{1}{\log m} \Big), 
  \end{equation*}
where $X_m \sim \operatorname{Exp}(\log m)$.
\end{theorem}

This is somewhat reminiscent of the situation for the number of parts in a \emph{partition}: Goh and Schmutz \cite{Goh-Schmutz:1995:number} proved that the number of distinct parts in a random partition of a large integer $n$ is asymptotically normally distributed, while the total number of parts with multiplicity was shown by Erd\H{o}s and Lehner \cite{Erdos-Lehner:1941:distribution} to follow a Gumbel distribution (which can also be represented as a sum of shifted exponential random variables). Indeed, our results are multiplicative analogues in a certain sense, since products turn into sums upon applying the logarithm, and $\log v_n(\phi,\phibar) \sim (\log \phi) n$.

We remark that the monoid $\calM(\calF)$ fits the definition of an \emph{arithmetical semigroup} as studied in abstract analytic number theory (see \cite{Knopfmacher:1975} for a general reference on the subject). However, as Theorem~\ref{theorem:count} shows, it is very sparse, so it does not satisfy the growth conditions that are typically imposed in this context. For arithmetical semigroups that satisfy such growth conditions, Erd\H{o}s-Kac type theorems are known as well, see \cite[Theorem 7.6.5]{Knopfmacher-Zhang:2001} and \cite[Theorem 3.1]{Wehmeier:2007}.

\section{Proof of Theorems~\ref{theorem:count} and \ref{theorem:asymptotic-normality}}

For real $u$ in a neighbourhood of $1$, we consider the Dirichlet generating function $d(z, u)$ that is defined as follows:
\begin{equation*}
  d(z, u):=\sum_{n\in\calM(\calF)}\frac{u^{\omega_\calF(n)}}{n^z},
\end{equation*}
for all complex $z$  for which the series converges (Lemma~\ref{lemma:central-approximation} will provide detailed information on convergence). Within the region of convergence, we have the product representation
\begin{equation}
  d(z, u)=\prod_{m\in \calF}(1+um^{-z}+um^{-2z}+\cdots) = \prod_{m\in\calF}\Bigl(1+\frac{um^{-z}}{1-m^{-z}}\Bigr).\label{eq:d-factorisation} 
\end{equation}

Set $h(n)=u^{\omega_\calF(n)}$ if $n\in\calM(\calF)$ and $h(n) = 0$ otherwise. We use the Mellin--Perron summation formula in the version
\begin{equation}\label{eq:mellin-perron}
  \sum_{1\le n\le x}h(n)\Bigl(1-\frac{n}{x}\Bigr)=\frac{1}{2\pi i}\int_{r-i\infty}^{r+i\infty} \biggl(\sum_{n=1}^{\infty}\frac{h(n)}{n^z}\biggr)x^z\frac{dz}{z(z+1)},
\end{equation}
for $x>0$ where $r$ is in the half-plane of absolute convergence of the
Dirichlet series $\sum_{n=1}^\infty \frac{h(n)}{n^s}$ (see e.g. \cite[Chapter 13]{Apostol:1976:introduction} or \cite[Theorem 2.1]{Flajolet-Grabner-Kirschenhofer-Prodinger:1994:mellin}). Thus
\begin{equation}\label{eq:Mellin-Perron-1}
  I_{\omega_\calF}(x, u):= \sum_{\substack{n\in\calM(\calF)\\n\le x}} u^{\omega_\calF(n)}\Bigl(1-\frac{n}{x}\Bigr) =
  \frac1{2\pi i}\int_{r-i\infty}^{r+i\infty}\frac{d(z, u)}{z(z+1)}x^z\,dz.
\end{equation}

We will use a saddle-point approach to evaluate this integral. As a first step,
we establish an estimate for $d(z,u)$ for $z=r+it$, $r\to 0^+$ and small
$t$. We start with an auxiliary lemma for another Dirichlet generating series.

\begin{lemma}\label{lemma:Dirichlet-series}
  Let $v$ be a complex number such that $\abs{v}<v_0$, where
  \begin{equation}\label{eq:definition-v_0}
    v_0 := \min \Bigl\{\log m \Bigm| m\in\calF_0 \text{ or } m=\frac{\phi^{13}}{\phi-\phibar}\Bigr\}
  \end{equation}
  and let $\Lambda(s, v)$ be given by the Dirichlet series
  \begin{equation*}
    \Lambda(s, v):=\sum_{m\in\calF}\frac1{(\log m - v)^s},
  \end{equation*}
which converges for $\Re s > 1$. The function  $\Lambda(s, v)$ can be analytically continued to a meromorphic function in $s$ with a single simple
  pole at $s=1$ with residue $1/\log\phi$.

  We have
  \begin{align}
    \Lambda(0, v)&=\abs{\calF_0}-\frac{25}{2} + \frac{\log(\phi-\phibar)+v}{\log\phi},\label{eq:Lambda-0-formula}\\
    \frac{\partial\Lambda(s, v)}{\partial s}\Bigr\rvert_{s=0}&= - \sum_{m\in \calF} \Big( \log \Big( 1 - \frac{v}{\log m} \Big) + \frac{v}{\log m} \Big) + \kappa_1v + \kappa_2,\notag
  \end{align}
where $\kappa_1$ and $\kappa_2$ are constants, and $\kappa_1$ is given by
\begin{equation}
\kappa_1 := \frac{\gamma - \log \log \phi}{\log \phi} + \sum_{m \in \calF_0} \frac{1}{\log m} + \frac{1}{\log v_{13}(\phi,\phibar)} + \sum_{k \geq 1} \Big( \frac{1}{\log v_{k+13}(\phi,\phibar)} - \frac{1}{k \log \phi} \Big).
\end{equation}
  Finally,
  \begin{equation*}
    \Lambda(s, v)=O(s^2)
  \end{equation*}
  for $-3/2\le\Re s\le 2$ with $\abs{s-1}\ge 1$ and $\abs{s}\ge 1$.
\end{lemma}
\begin{proof}
  Let
  \begin{align*}
    \Lambda_0(s, v)&:=\sum_{m\in\calF_0} \frac{1}{(\log m - v)^s},\\
    \Lambda_1(s, v)&:=\sum_{j\ge
      13}\Bigl(\Bigl(\log\frac{\phi^j-\phibar^{j}}{\phi-\phibar} -
      v\Bigr)^{-s} - \Bigl(\log\frac{\phi^j}{\phi-\phibar} -
      v\Bigr)^{-s}\Bigr),\\
    \alpha(v)&:= 13 -\frac{\log(\phi-\phibar)+v}{\log\phi}.
  \end{align*}
  Note that our choice of $v_0$ guarantees that each summand of $\Lambda$,
  $\Lambda_0$ and $\Lambda_1$ is well-defined and that $\Re \alpha(v)>0$.

  Thus
\begin{align*}
  \Lambda(s, v)&= \Lambda_0(s, v)+\sum_{j\ge
    13}\frac{1}{(j\log\phi-\log(\phi-\phibar)-v)^s}+\Lambda_1(s, v)\\
  &= \frac{1}{\log^s\phi}\zeta(s, \alpha(v)) + \Lambda_0(s, v) + \Lambda_1(s, v),
\end{align*}
  where $\zeta(s, \beta)=\sum_{k\ge 0}(k+\beta)^s$ denotes the Hurwitz zeta
  function (the series converges for $\Re s  > 1$ if $\beta \not\in \{0,-1,-2,\ldots\}$, and it can be analytically continued),  cf.~\DLMF{25.11}{1}.

  The function $\Lambda_0(s, v)$ is obviously an entire function, and it is bounded
  for $\Re s\ge -3/2$. 

  Estimating the difference occurring in $\Lambda_1(s, v)$, we see that
  \begin{equation*}
    \Lambda_1(s, v)=O\biggl(\sum_{j\ge 13}\frac{s(\abs\phibar/\phi)^{j}}{(j\log\phi-\log(\phi-\phibar)-v)^{\Re s+1}}\biggr).
  \end{equation*}
  Thus $\Lambda_1(s,v)$ is an entire function, and we have the estimate
  $\Lambda_1(s, v)=O(s)$ for $\Re s\ge -3/2$.

Moreover, $\zeta(s, \alpha)$ is a meromorphic function  with a single simple
  pole at $s=1$ with residue $1$, and by~\cite[\S 13.51]{Whittaker-Watson:1996}, we have
  \begin{equation*}
    \zeta(s, \alpha(v))=O(s^2)
  \end{equation*}
  in the given area, which proves the desired asymptotic estimate. It remains to determine the values of the function and its derivative at $s=0$.

Observe first that $\Lambda_0(0, v)=\abs{\calF_0}$ and $\Lambda_1(0, v)=0$.
Moreover, we have
  \begin{equation*}
    \zeta(0, \alpha(v))=\frac12-\alpha(v),
  \end{equation*}
 cf.~\DLMF{25.11}{13}. Now we consider the first derivative. $\Lambda_0$ and $\Lambda_1$ are simply differentiated term by term:
\begin{equation*}
 \frac{\partial\Lambda_0(s, v)}{\partial s}\Bigr\rvert_{s=0}  = - \sum_{m \in \calF_0} \log(\log m - v) = - \sum_{m \in \calF_0} \Big( \log \log m +  \log \Big( 1 - \frac{v}{\log m} \Big) \Big)
\end{equation*}
and
\begin{align*}
 \frac{\partial\Lambda_1(s, v)}{\partial s}\Bigr\rvert_{s=0}  &= - \sum_{j \geq 13} \Bigl(\log \Bigl(\log\frac{\phi^j-\phibar^{j}}{\phi-\phibar} -
      v\Bigr) - \log \Bigl(\log\frac{\phi^j}{\phi-\phibar}  - v\Bigr)\Bigr) \\
&= - \sum_{j \geq 13} \Bigl(\log \Bigl(\log v_j(\phi,\phibar) -
      v\Bigr) - \log \log \phi - \log(j-13 + \alpha(v))\Bigr).
\end{align*}
Finally, it is well known (see~\DLMF{25.11}{18}) that
\begin{equation*}
\frac{\partial \zeta(s,\alpha)}{\partial s}\Bigr\rvert_{s=0} = \log \Gamma(\alpha) - \frac12 \log (2\pi),
\end{equation*}
hence
\begin{equation*}
\frac{\partial}{\partial s} \frac{1}{\log^s\phi}\zeta(s, \alpha(v)) \Bigr\rvert_{s=0} = - \log \log \phi \Big( \frac12 - \alpha(v) \Big) + \log \Gamma(\alpha(v)) - \frac12 \log (2\pi).
\end{equation*}
Now we make use of the product representation of the Gamma function, which yields
\begin{align*}
\frac{\partial}{\partial s} \frac{1}{\log^s\phi}\zeta(s, \alpha(v)) \Bigr\rvert_{s=0} &= - \log \log \phi \Big( \frac12 - \alpha(v) \Big) - \frac12 \log (2\pi) - \gamma \alpha(v) - \log \alpha(v) \\
&\quad - \sum_{k \geq 1} \Big( \log(k + \alpha(v)) - \log k - \frac{\alpha(v)}{k} \Big).
\end{align*}
The infinite sum can be combined with the sum in the derivative of $\Lambda_1$ to give us
\begin{align*}
 \frac{\partial\Lambda(s, v)}{\partial s}\Bigr\rvert_{s=0}&= - \sum_{m \in \calF_0} \Big( \log \log m +  \log \Big( 1 - \frac{v}{\log m} \Big) \Big) - \log \log \phi \Big( \frac12 - \alpha(v) \Big) \\
&\quad- \frac12 \log(2\pi) - \gamma \alpha(v) - \log \Big( 1 - \frac{v}{\log  v_{13}(\phi,\phibar)} \Big)\\
&\quad - \log \log v_{13}(\phi,\phibar) + \log \log \phi \\
&\quad- \sum_{k \geq 1}  \Big( \log(\log v_{k+13}(\phi,\phibar) - v) - \log \log \phi - \log k - \frac{\alpha(v)}{k} \Big),
\end{align*}
which can finally be rewritten as
\begin{align*}
 \frac{\partial\Lambda(s, v)}{\partial s}\Bigr\rvert_{s=0}&= - \sum_{m\in \calF} \Big( \log \Big( 1 - \frac{v}{\log m} \Big) + \frac{v}{\log m} \Big) \\
&\quad - \sum_{m \in \calF_0} \log \log m + \sum_{m \in \calF_0} \frac{v}{\log m} + \log \log \phi \Big( \frac{27}{2} - \frac{\log(\phi - \phibar) + v}{\log \phi} \Big) \\
&\quad - \frac12 \log(2\pi) - \gamma \Big( 13 - \frac{\log(\phi - \phibar) + v}{\log \phi} \Big) - \log \log v_{13}(\phi,\phibar) \\
&\quad + \frac{v}{\log v_{13}(\phi,\phibar)} + v \sum_{k \geq 1} \Big( \frac{1}{\log v_{k+13}(\phi,\phibar)} - \frac{1}{k \log \phi} \Big) \\
&\quad - \sum_{k \geq 1} \Big( \log \log v_{k+13}(\phi,\phibar) - \log \log \phi - \log k - \frac{13}{k} + \frac{\log(\phi-\phibar)}{k \log \phi} \Big).
\end{align*}
Apart from the first sum, all terms are indeed either constant or linear in $v$. Collecting the linear terms gives the stated formula for $\kappa_1$, completing our proof.
\end{proof}

\begin{lemma}\label{lemma:central-approximation}
  Let $r>0$, $z=r+it$ with $\abs{t}\le r^{7/5}$, and $\abs{1-u}<1$. 

The Dirichlet series $d(z, u)$ converges absolutely for $\Re z>0$, and we have the asymptotic estimates
  \begin{align}
   d(z, u) &= d(r, u)\exp\Bigl( -\frac{ia(u)t}{r^2} - \frac{a(u)t^2}{r^3} +
   O(r^{1/5})\Bigr),\notag\\
   d(r, u) &= \exp\Bigl(\frac{a(u)}{r} + b\log r +c(u) +  O(r)\Bigr)\label{eq:d-r-u-estimate}
  \end{align}
  for $r\to 0^+$ and
  \begin{equation*}
    a(u)=\frac{\pi^2/6-\Li_{2}(1-u)}{\log \phi},
  \end{equation*}
  where $\Li$ denotes the polylogarithm, 
  \begin{equation*}
    b=-\abs{\calF_0}+\frac{25}{2} - \frac{\log(\phi-\phibar)}{\log\phi}
  \end{equation*}
  and  $c(u)$
  is a function which is analytic in a complex neighbourhood of $1$. The estimates are uniform in $u$ on compact sets.
\end{lemma}
\begin{proof}
Let
\begin{equation*}
  g(z, u):=\sum_{m\in\calF}\log\Bigl(1+\frac{um^{-z}}{1-m^{-z}}\Bigr),
\end{equation*}
which implies $d(z, u)=\exp(g(z,u))$ by \eqref{eq:d-factorisation}.
For $\Re z>0$ and fixed $u$, we have
\begin{equation*}
  \log\Bigl(1+\frac{um^{-z}}{1-m^{-z}}\Bigr) \sim um^{-z}
\end{equation*}
for $m\to\infty$. Since the elements of $\calF$ follow the asymptotic formula $v_\ell = v_\ell(\phi,\phibar) \sim \frac{\phi^\ell}{\phi-\phibar}$, the
series $g(z,u)$ converges absolutely.

We rewrite $g(z, u)$ as
\begin{align*}
  g(z, u)= \sum_{m\in\calF}\log\Bigl(1+\frac{ue^{-z\log m}}{1-e^{-z\log m}}\Bigr)=\sum_{m\in\calF}
f(z\log m, u)
\end{align*}
for $f(z, u)=\log\bigl(1+\frac{ue^{-z}}{1-e^{-z}}\bigr)$.

For $\Re s>1$, we consider the Mellin transform $g^\star(s, u)$ of the harmonic sum $g(z, u)$
and obtain
\begin{equation}\label{eq:g-star}
  g^\star(s, u) = f^\star(s, u) \sum_{m\in\calF}\frac1{(\log m)^s}=f^\star(s,
  u)\Lambda(s, 0)
\end{equation}
with $\Lambda$ as defined in Lemma~\ref{lemma:Dirichlet-series}.

We compute $f^\star(s, u)$. We have
\begin{align*}
  f(z, u)&
  =\log\Bigl(1+\frac{ue^{-z}}{1-e^{-z}}\Bigr)
  =\log(1-(1-u)e^{-z}) - \log(1-e^{-z})\\
  &=\sum_{j=1}^{\infty}\frac{1-(1-u)^j}{j}e^{-zj}.
\end{align*}
Thus the Mellin transform is
\begin{equation*}
  f^\star(s, u) = \biggl(\sum_{j\ge 1}\frac{1-(1-u)^j}{j^{1+s}}\biggr)\Gamma(s)
  =(\zeta(s+1)-\Li_{s+1}(1-u)) \Gamma(s).
\end{equation*}
Note that the polylogarithm $\Li_{s+1}(1-u)$ is an entire
function in $s$ for $\abs{1-u}<1$.

We conclude that
\begin{equation*}
  g^\star(s, u) = (\zeta(s+1)-\Li_{s+1}(1-u))\Gamma(s)\Lambda(s, 0).
\end{equation*}
This is a meromorphic function in $s$ with a simple pole at $s=1$, a double pole at
$s=0$ and at most simple poles when $s$ is a negative integer. When $s$ runs along vertical lines,
$g^\star(s, u)z^{-s}$ decreases exponentially for $\abs{\arg(z)}<\pi/4$ due to the factor $\Gamma(s)$.
We have
\begin{equation*}
  \Res_{s=1} g^\star(s, u)=\frac{\pi^2/6-\Li_{2}(1-u)}{\log \phi} =: a(u).
\end{equation*}
In addition, we have the 
singular expansion
\begin{equation*}
  g^\star(s, u) = -\frac{b}{s^2} + \frac{c(u)}{s} + O(1)
\end{equation*}
around $0$, where $b=-\Lambda(0, 0)$ and 
\begin{equation*}
c(u) = \Lambda(0,0) \log u + \frac{\partial \Lambda(s,0)}{\partial s} \Bigr\rvert_{s=0}
\end{equation*}
is an analytic function of $u$ for $\abs{1-u}<1$. Note that  $b$ is independent of $u$ because the only component
of $g^\star(s,u)$ depending on $u$ is $\Li_{s+1}(1-u)$, which does not
contribute to the term $1/s^2$.

By \cite[Theorem~4]{Flajolet-Gourdon-Dumas:1995:mellin} (which remains valid
for complex $z$ with $\abs{\arg(z)}<\pi/4$ because of the exponential decay
observed above; cf.~\cite{Flajolet-Prodinger:1986:regis}), we get
\begin{equation*}
  g(z, u) = \frac{a(u)}{z}+b\log z+c(u)+ O(z)
\end{equation*}
for $z\to 0$, $\abs{\arg(z)}<\pi/4$.

As the Mellin transform of $z \frac{\partial g(z, u)}{\partial z}$ is
$(-s)g^\star(s, u)$ by general properties of the Mellin transform, we
immediately deduce that
\begin{equation}\label{eq:first-derivative-g}
  z \frac{\partial g(z, u)}{\partial z}=-\frac{a(u)}{z}+O(1)
\end{equation}
for $z\to 0$, $\abs{\arg(z)}<\pi/4$, because the Mellin transform now has a simple pole at $s=0$. Repeating the argument, we get
\begin{align}
  z^2 \frac{\partial^2 g(z, u)}{\partial^2 z} + z \frac{\partial g(z,
    u)}{\partial z}&=\frac{a(u)}{z}+O(z),\label{eq:second-derivative-g}\\
  z^3\frac{\partial^3 g(z, u)}{\partial^3 z} + 3z^2 \frac{\partial^2 g(z, u)}{\partial^2 z} + z \frac{\partial g(z,
    u)}{\partial z}&=-\frac{a(u)}{z}+O(z)\label{eq:third-derivative-g}
\end{align}
for $z\to 0$, $\abs{\arg(z)}<\pi/4$.
Solving the linear system consisting of \eqref{eq:first-derivative-g},
\eqref{eq:second-derivative-g} and \eqref{eq:third-derivative-g} yields
\begin{align*}
  \frac{\partial g(z, u)}{\partial z}&=-\frac{a(u)}{z^2}+O\Bigl(\frac1z\Bigr),\\
  \frac{\partial^2 g(z, u)}{\partial^2 z}
  &=\frac{2a(u)}{z^3}+O\Bigl(\frac1{z^2}\Bigr),\\
  \frac{\partial^3 g(z, u)}{\partial^3 z}
  &=-\frac{6a(u)}{z^4}+O\Bigl(\frac1{z^3}\Bigr)
\end{align*}
for $z\to 0$, $\abs{\arg(z)}<\pi/4$.

Thus we can approximate $g(z, u)$ by Taylor expansion around $z=r$ as
\begin{equation*}
  g(z, u)=g(r, u) -\frac{ia(u)t}{r^2}+O\Bigl(\frac{t}{r}\Bigr) - \frac{a(u)t^2}{r^3} +
  O\Bigl(\frac{t^2}{r^2}\Bigr) + O\Bigl(\frac{t^3}{r^4}\Bigr).
\end{equation*}
With our choice of the upper bound for $t$, we get
\begin{equation*}
    g(z, u)=g(r, u) -\frac{ia(u)t}{r^2} - \frac{a(u)t^2}{r^3}
  + O(r^{1/5}).
\end{equation*}
This concludes the proof of the lemma.
\end{proof}

As a next step, we show that  $\abs{d(r+it,u)}$ is exponentially smaller
than $d(r,u)$ for all $t$ which are not too close to integer multiples of $2\pi/\log\phi$.

\begin{lemma}\label{lemma:medium-range}
  Let $r>0$ and $z=r+it$ with $\abs{t}\ge r^{7/5}$. Then
  \begin{equation*}
    \log d(r, u) - \Re \log d(z, u)\gg \frac1{r^{1/5}}
  \end{equation*}
  for $u\in (1/2, 3/2)$ and $r\to 0^+$ unless there is a non-zero integer $k$ such that
  $\abs{t-2k\pi/\log \phi}<r^{3/4}$.
\end{lemma}
\begin{proof}
  Using the function $g(z,u)$ from the beginning of the proof of
  Lemma~\ref{lemma:central-approximation}, we have
  \begin{align*}
    g(z, u)&=\sum_{m\in\calF}
    \log\Bigl(\frac{1-(1-u)m^{-z}}{1-m^{-z}}\Bigr)
    =\sum_{m\in\calF}\sum_{k\ge 1}\frac{1-(1-u)^k}{k}m^{-kz}\\
    &=\sum_{m\in\calF}\sum_{k\ge 1}\frac{1-(1-u)^k}{k}m^{-kr}(\cos(kt\log
    m)-i\sin(k t\log m)).
  \end{align*}
  This implies that
  \begin{equation*}
    \log d(r, u) - \Re \log d(z, u) = \sum_{m\in\calF}\sum_{k\ge 1}\frac{1-(1-u)^k}{k}m^{-kr}(1-\cos(kt\log
    m)).
  \end{equation*}
  Obviously, all summands are non-negative, so we take the first summand as a
  lower bound and obtain
  \begin{equation*}
    \log d(r, u) - \Re \log d(z, u) \ge u\sum_{m\in\calF}m^{-r}(1-\cos(t\log m)).
  \end{equation*}
  For $\ell\ge 13$, we use the estimate
  $v_\ell=\frac{\phi^\ell-\phibar^\ell}{\phi-\phibar}\le K_0 \phi^\ell$ for a
  suitable $K_0>1$. In the following, we assume that $r<1$. If $13\le \ell\le 1/r$, then
  \begin{equation*}
    v_\ell^{r}\le  K_0^r\phi^{\ell r} \le K_0\phi.
  \end{equation*}

  Thus restricting the sum to those $\ell$ and the corresponding $v_\ell$ yields
  \begin{equation}\label{eq:medium-range-simpler-sum}
    \log d(r, u) - \Re \log d(z, u) \gg \sum_{13\le \ell\le 1/r}(1-\cos(t\log v_\ell)).
  \end{equation}
  We first consider the case that $\abs{t}\le r/\log\phi$. In this case, we have
  \begin{equation*}
    \abs{t}\log v_\ell\le \frac{r}{\log\phi} (\ell \log\phi + \log K_0) \le 1 + \frac{r\log K_0}{\log\phi}<\frac{\pi}{2}
  \end{equation*}
  for sufficiently small $r$. Thus we may use the inequality $1-\cos \theta=2\sin^2(\theta/2)\ge (4/\pi^2) \theta^2$
  (which is a consequence of concavity of $\sin$) to obtain
  \begin{equation*}
    \log d(r, u) - \Re \log d(z, u) \gg \sum_{13\le \ell\le 1/r}t^2\log^2
    v_\ell \gg t^2 \sum_{13\le \ell\le 1/r} \ell^2\gg \frac{t^2}{r^3}\ge \frac1{r^{1/5}}.
  \end{equation*}

  Next, we consider the case that $r/\log\phi\le\abs{t}\le r^{1/5}$. Here we omit
  all summands with $\ell> 1/(\abs{t}\log\phi)$ in \eqref{eq:medium-range-simpler-sum} and
  obtain
  \begin{align*}
    \log d(r, u) - \Re \log d(z, u) &\gg \sum_{13\le \ell\le
      1/(\abs{t}\log\phi)}(1-\cos(t\log v_\ell)) \\
    &\gg t^2\sum_{13\le \ell\le
      1/(\abs{t}\log\phi)} \log^2v_\ell \gg \frac{t^2}{\abs{t}^3}=\frac1{\abs{t}}\ge \frac{1}{r^{1/5}}
  \end{align*}
  by the same arguments as above.

  Finally, we turn to the case that $\abs{t}> r^{1/5}$.
  We estimate the sum of the cosines by a geometric sum:
  \begin{align*}
    \sum_{13\le \ell \le 1/r}\cos(t\log v_\ell) &=
    \Re \sum_{13\le \ell \le 1/r}\exp(it\log v_\ell) \\
    &\le
    \abs[\bigg]{ \sum_{13\le \ell \le 1/r}\exp\biggl(it\Bigl(\ell\log\phi - \log(\phi-\phibar) + O\Bigl(\Bigl(\frac{\abs{\phibar}}{\phi}\Bigr)^{\ell}\Bigr)\Bigr)\biggr)}\\
    &=
    \abs[\bigg]{ \sum_{13\le \ell \le 1/r}\Bigl(\exp(it\ell\log\phi) + O\Bigl(\Bigl(\frac{\abs{\phibar}}{\phi}\Bigr)^{\ell}\Bigr)\Bigl)}\\
    &\le \frac{2}{\abs{\exp(it\log\phi)-1}}+O(1).
  \end{align*}
  Choose an integer $k$ such that $\abs{t\log\phi -2k\pi}$ is minimal.
  Then $\abs{t\log\phi-2k\pi}\ge r^{3/4}\log\phi$ in view of the assumption made on $t$ in the statement of the lemma
  (if $k\neq 0$) and because $\abs{t}>r^{1/5}$ (if $k=0$). We therefore have $\abs{\exp(it\log\phi)-1}\gg
  r^{3/4}$. Combining this with \eqref{eq:medium-range-simpler-sum}, we
  conclude that
  \begin{equation*}
    \log d(r, u) - \Re \log d(z, u) \gg \frac1r-\frac1{r^{3/4}}\gg \frac1r,
  \end{equation*}
  as required.
\end{proof}

We are now able to estimate the integral in \eqref{eq:Mellin-Perron-1},
choosing $r$ appropriately.

\begin{lemma}\label{lemma:Mellin-Perron-integral}
  Let $u\in(1/2, 3/2)$ and $r=\sqrt{a(u)/\log x}$. Then
  \begin{equation}\label{eq:mellin-perron-integral}
    \frac1{2\pi i}\int_{r-i\infty}^{r+i\infty}\frac{d(z, u)}{z(z+1)}x^z\, dz =
  \frac{d(r, u)x^r}{2\pi}\frac{r^{1/2}\sqrt{\pi}}{\sqrt{a(u)}}(1+O(r^{1/5}))
  \end{equation}
  for $x\to\infty$.
\end{lemma}
\begin{proof}
We first compute the integral over
\begin{equation*}
  M_0=\{r+it \mid \abs{t}\le r^{7/5}\}.
\end{equation*}
For  $z\in M_0$, we have
\begin{equation*}
  \frac{1}{z}=\frac1r\Bigl(1+O\Bigl(\frac tr\Bigr)\Bigr)=\frac1r(1+O(r^{2/5})), \qquad
  \frac{1}{z+1}=1+O(r).
\end{equation*}
By Lemma~\ref{lemma:central-approximation}, we have
\begin{equation*}
  \frac{d(z, u) x^z}{z(z+1)} = \frac{d(r, u)x^r}{r}\exp\Bigl( -\frac{ia(u)t}{r^2} - \frac{a(u)t^2}{r^3}
  +it\log x + O(r^{1/5})\Bigr).
\end{equation*}
The value $r=\sqrt{a(u)/\log x}$ has been chosen in such a way that the linear terms
vanish. Thus we get
\begin{equation*}
  \frac{d(z, u) x^z}{z(z+1)} = \frac{d(r, u)x^r(1+O(r^{1/5}))}{r}\exp\Bigl( - \frac{a(u)t^2}{r^3}\Bigr).
\end{equation*}
We have
\begin{align*}
  \frac1{2\pi i}\int_{z\in M_0}\frac{d(z, u)}{z(z+1)}x^z\, dz &=
  \frac{d(r, u)x^r(1+O(r^{1/5}))}{2\pi r} \int_{-r^{7/5}}^{r^{7/5}} \exp\Bigl( -
  \frac{a(u)t^2}{r^3}\Bigr)\,dt\\
  &=\frac{d(r, u)x^r(1+O(r^{1/5}))}{2\pi r} \int_{-\infty}^\infty \exp\Bigl( -
  \frac{a(u)t^2}{r^3}\Bigr)\,dt
\end{align*}
because adding the tails induces an exponentially small error (note that
$a(u)>0$). Computing the integral yields
\begin{equation}\label{eq:integral-central-region}
  \frac1{2\pi i}\int_{z\in M_0}\frac{d(z, u)}{z(z+1)}x^z\, dz =
  \frac{d(r, u)x^r}{2\pi}\frac{r^{1/2}\sqrt{\pi}}{\sqrt{a(u)}}(1+O(r^{1/5})).
\end{equation}

Next, we compute the integral over
\begin{equation*}
  M_1 = \{r+it\in\C \mid \exists k\in\Z\setminus\{0\}\colon  \abs{t-2k\pi/\log \phi}<r^{3/4}\}.
\end{equation*}
We use the trivial bound $\abs{d(z, u)}\le d(r, u)$, which follows from the
definition of $d(z, u)$ as a Dirichlet series. Using the estimates
$\abs{z}\ge\abs{\Im z}$ and $\abs{z+1}\ge\abs{\Im z}$ yields
\begin{align*}
  \abs[\bigg]{\frac1{2\pi i}\int_{z\in M_1} \frac{d(z, u)}{z(z+1)}x^z\,dz}&\le
  \frac{d(r, u)x^r}{2\pi}\sum_{k\in\Z\setminus\{0\}}
  \frac{2r^{3/4}}{\bigl(\abs{\frac{2k\pi}{\log\phi}}-r^{3/4}\bigr)^2}\\
  &\ll \frac{d(r, u)x^r}{2\pi} r^{3/4}.
\end{align*}
Thus this integral can be absorbed by the error term of
\eqref{eq:integral-central-region}.

Finally, we compute the integral over
\begin{equation*}
  M_2:=\{r+it\in\C\} \setminus (M_0 \cup M_1).
\end{equation*}
For $z\in M_2$, we have
\begin{equation*}
  \abs{d(z, u)}\le d(r, u)\exp\Bigl(-\frac{K_1}{r^{1/5}}\Bigr)
\end{equation*}
for a suitable positive constant $K_1$ by Lemma~\ref{lemma:medium-range}. Thus
\begin{align*}
 \abs[\bigg]{\frac1{2\pi i}\int_{z\in M_2} \frac{d(z, u)}{z(z+1)}x^z\,dz} &\ll d(r, u)x^r\exp\Bigl(-\frac{K_1}{r^{1/5}}\Bigr) \int_{z \in M_2} \frac{1}{|z(z+1)|}\,d|z| \\
&\ll d(r, u)x^r\exp\Bigl(-\frac{K_1}{r^{1/5}}\Bigr) \log \frac1r.
\end{align*}
As this integral is also absorbed by the error term of
\eqref{eq:integral-central-region}, we get \eqref{eq:mellin-perron-integral}.
\end{proof}

In view of \eqref{eq:Mellin-Perron-1}, Lemma~\ref{lemma:Mellin-Perron-integral} immediately gives us
\begin{equation*}
  I_{\omega_\calF}(x, u) = \sum_{\substack{n\in\calM(\calF)\\n\le x}} u^{\omega_\calF(n)}\Bigl(1-\frac{n}{x}\Bigr) \sim \frac{d(r, u)x^r}{2\pi}\frac{r^{1/2}\sqrt{\pi}}{\sqrt{a(u)}}.
\end{equation*}
However, we are actually interested in an expression for the sum without the additional factor $(1-\frac{n}{x})$. This is achieved in the following lemma.

\begin{lemma}\label{lemma:gf-lower}
  We have
  \begin{multline*}
    \sum_{\substack{n\in\calM(\calF)\\n\le x}}u^{\omega_\calF(n)} =\\\frac{1}{2\sqrt\pi}\exp\Bigl(2\sqrt{a(u)}\sqrt{\log x}-\frac{2b+1}{4}\log\log
  x+\frac{2b-1}{4}\log
  a(u)
  + c(u)  \Bigr)\\
  \times\Bigl(1+O\Bigl(\frac{1}{(\log x)^{1/10}}\Bigr)\Bigr)
  \end{multline*}
  for $x\to\infty$ and $1/2<u<3/2$.
\end{lemma}
\begin{proof}
Trivially, the inequality
\begin{equation}\label{eq:gf-lower-bound}
 I_{\omega_\calF}(x, u) \le
    \sum_{\substack{n\in\calM(\calF)\\n\le x}} u^{\omega_\calF(n)}
\end{equation}
holds for positive $u$. On the other hand, we also have
\begin{equation}\label{eq:gf-upper-bound}
\begin{aligned}
  \frac{I_{\omega_\calF}(x\log x, u)}{1-\frac1{\log x}}&=
  \frac1{1-\frac1{\log x}} \sum_{\substack{n\in\calM(\calF)\\n\le x\log x}}
  u^{\omega_\calF(n)}\Bigl(1-\frac{n}{x\log x}\Bigr)\\
  &\ge
\frac1{1-\frac1{\log x}} \sum_{\substack{n\in\calM(\calF)\\n\le x}}
  u^{\omega_\calF(n)}\Bigl(1-\frac{n}{x\log x}\Bigr)\\
&\ge
\sum_{\substack{n\in\calM(\calF)\\n\le x}} u^{\omega_\calF(n)}
\end{aligned}
\end{equation}
for positive $u$ and $x>1$.

Now we choose $r=\sqrt{a(u)/\log x}$ as in Lemma~\ref{lemma:Mellin-Perron-integral}
and use \eqref{eq:Mellin-Perron-1} as well as Lemma~\ref{lemma:Mellin-Perron-integral} to obtain
\begin{equation*}
  I_{\omega_{\calF}}(x, u)=\frac{d(r, u)x^rr^{1/2}}{2\sqrt{\pi a(u)}}(1+O(r^{1/5})).
\end{equation*}
Finally, the asymptotic formula~\eqref{eq:d-r-u-estimate} for $d(r,u)$ gives us
\begin{align*}
  I_{\omega_{\calF}}(x, u)&=\frac{1}{2\sqrt\pi}\exp\Bigl( \frac{a(u)}{r}+b\log r
  + c(u) +r\log x + \frac12(\log r-\log a(u))\Bigr)\\&\qquad\qquad\times(1+O(r^{1/5}))\\
  &=\frac{1}{2\sqrt\pi}\exp\Bigl(2\sqrt{a(u)}\sqrt{\log x}-\frac{2b+1}{4}\log\log
  x+\frac{2b-1}{4}\log
  a(u)
  + c(u)  \Bigr)\\
  &\qquad\qquad\times\Bigl(1+O\Bigl(\frac{1}{(\log x)^{1/10}}\Bigr)\Bigr).
\end{align*}
If we replace $x$ by $x\log x$ and divide by $(1-1/\log x)$, we obtain
exactly the same asymptotic expansion, the difference being absorbed by the
error term. Combining this with \eqref{eq:gf-lower-bound} and
\eqref{eq:gf-upper-bound} yields the result.
\end{proof}

We are now able to prove Theorems~\ref{theorem:count} and \ref{theorem:asymptotic-normality}.

\begin{proof}[Proof of Theorem~\ref{theorem:count}]
  Setting $u=1$ in Lemma~\ref{lemma:gf-lower} yields the result.
\end{proof}

\begin{proof}[Proof of Theorem~\ref{theorem:asymptotic-normality}]
  We consider the moment generating function
  \begin{equation*}
    \E(e^{\omega_\calF(N)t}) = \frac{\sum_{\substack{n\in\calM(\calF)\\n\le
          x}}e^{\omega_\calF(n)t}}{\sum_{\substack{n\in\calM(\calF)\\n\le
          x}}1}.
  \end{equation*}
  Lemma~\ref{lemma:gf-lower} yields
  \begin{equation*}
    \E(e^{\omega_\calF(N)t})= \exp\bigl(2(\sqrt{a(e^t)}-\sqrt{a(1)})\sqrt{\log x}+O(t)\bigr)\Bigl(1+O\Bigl(\frac{1}{(\log x)^{1/10}}\Bigr)\Bigr)
  \end{equation*}
for $\log \frac12 < t < \log \frac32$.  We compute the Taylor expansion of $2(\sqrt{a(e^t)}-\sqrt{a(1)})$ around $0$ as
  \begin{equation*}
    2(\sqrt{a(e^t)}-\sqrt{a(1)}) = a_1 t + \frac{a_2}{2} t^2+O(t^3)
  \end{equation*}
  for the constants $a_1$, $a_2$ given in the theorem.
  Thus the moment generating function of the renormalised random variable
  $Z=(\omega_\calF(N)-a_1\log^{1/2} x)/(\sqrt{a_2}\log^{1/4}x)$ is
  \begin{equation*}
    \E(e^{Zt})=\exp\Bigl( \frac{t^2}2 +
    O\Bigl(\frac{t^3+t}{\log^{1/4}x}\Bigr) \Bigr)\Bigl(1+O\Bigl(\frac{1}{(\log x)^{1/10}}\Bigr)\Bigr).
  \end{equation*}
  For all real $t$, this moment generating function
  converges pointwise to the moment generating function $e^{t^2/2}$ of the
  standard normal distribution. By Curtiss' theorem~\cite{Curtiss:1942}, the
  random variable $Z$ converges weakly to the standard normal distribution for $x\to\infty$.
\end{proof}

\section{Proof of Theorem~\ref{theorem:distribution-Omega}: Counting with Multiplicities}

Let $r>0$ and consider the interval $U=U(r)=(\exp(-v_0r/2),  \exp(v_0r/2))$,
where $v_0$ has been defined in \eqref{eq:definition-v_0}.

For $u\in U$ and $\Re z>r/2$, we study the Dirichlet generating function
\begin{equation*}
    D(z, u)=\sum_{n\in\calM(\calF)}\frac{u^{\Omega_\calF(n)}}{n^z},
\end{equation*}
which has a product representation
\begin{equation*}
 D(z, u)=\prod_{m\in\calF}(1+um^{-z}+u^2m^{-2z}+\cdots) = \prod_{m\in\calF}\frac1{1-um^{-z}}
\end{equation*}
for all such $z$ and $u$.

\begin{lemma}\label{lemma:central-approximation-D}
  Let $r>0$, $\abs{t}\le r^{7/5}$ and $u\in U$. Then
  \begin{align}
   D(r+it, u) &= D(r, u)\exp\Bigl( -\frac{iAt}{r^2} - \frac{At^2}{r^3} +
   O(r^{1/5})\Bigr),\notag\\
   D(r, u) &= \exp\Bigl(\frac{A}{r} + B\Bigl(\frac{\log u}{r}\Bigr)\log r + C\Bigl(\frac{\log u}{r}\Bigr) +  O(r)\Bigr)\label{eq:D-r-u-estimate}
  \end{align}
  for $A=\pi^2/(6 \log\phi)$, $B(v)=-\Lambda(0, v)$, and $C(v)=\frac{\partial\Lambda(s, v)}{\partial s}\Bigr\rvert_{s=0}$ in a neighbourhood of the origin.
\end{lemma}
\begin{proof}
  Let $\Re z>r/2$ and $v=(\log u)/z$. The assumption $u\in U$ implies that $\abs{v}< rv_0/(2\abs
  z)< v_0$.

  We consider the sum
  \begin{equation*}
    G(z, v) = -\sum_{m\in\calF}\log(1-e^{vz}m^{-z}).
  \end{equation*}
  Expanding the logarithm yields
  \begin{equation*}
    G(z, v)=\sum_{m\in\calF}\sum_{k\ge 1}\frac{e^{kvz}m^{-kz}}{k} =
    \sum_{m\in\calF}\sum_{k\ge 1}\frac{\exp(-(\log m-v)kz)}{k}.
  \end{equation*}
  Its Mellin transform is
  \begin{equation*}
    G^\star(s, v) = \sum_{m\in\calF}\sum_{k\ge 1}
    \frac{1}{k^{1+s}}\frac1{(\log m-v)^s}\Gamma(s)=
    \Gamma(s)\zeta(1+s)\Lambda(s, v),
  \end{equation*}
  where $\Lambda$ has been defined in Lemma~\ref{lemma:Dirichlet-series}.

  Again, $G^\star(s, v)$ has a simple pole at $s=1$ and a double pole at $s=0$. At $s=1$,
  the local expansion is
  \begin{equation*}
    G^\star(s, v) =  \frac{A}{s-1} + O(1).
  \end{equation*}
  The local expansion around $s=0$ is
  \begin{equation*}
    G^\star(s, v) =  -\frac{B(v)}{s^2}+\frac{C(v)}{s} + O(1),
  \end{equation*}
with $B(v)$ and $C(v)$ as in the statement of the lemma.

  The rest of the proof follows along the lines of the proof of Lemma~\ref{lemma:central-approximation}.
\end{proof}

\begin{lemma}\label{lemma:D-medium-range}
  Let $r>0$ and $z=r+it$ with $\abs{t}\ge r^{7/5}$. Then
  \begin{equation*}
    \log D(r, u) - \Re \log D(z, u)\gg \frac1{r^{1/5}}
  \end{equation*}
  for $u\in (\exp(-v_0r/2), \exp(v_0r/2))$ and $r\to 0^+$ unless there is a non-zero integer $k$ such that
  $\abs{t-2k\pi/\log \phi}<r^{3/4}$.
\end{lemma}
\begin{proof}
  We have
  \begin{align*}
    \Re \log D(z, u)&=-\Re\sum_{m\in\calF}
    \log(1-um^{-z})
    =\Re\sum_{m\in\calF}\sum_{k\ge 1}\frac{u^k}{k}m^{-kz}\\
    &=\sum_{m\in\calF}\sum_{k\ge 1}\frac{u^k}{k}m^{-kr}\cos(kt\log
    m).
  \end{align*}
  This implies that
  \begin{align*}
    \log D(r, u) - \Re \log D(z, u) &= \sum_{m\in\calF}\sum_{k\ge
      1}\frac{u^k}{k}m^{-kr}(1-\cos(kt\log m))\\&\ge u\sum_{m\in\calF}m^{-r}(1-\cos(t\log m)).
  \end{align*}
  The remainder of the proof is exactly the same as that of
  Lemma~\ref{lemma:medium-range}.
\end{proof}

\begin{proof}[Proof of Theorem~\ref{theorem:distribution-Omega}]
We now consider asymptotic
expansions for $x\to\infty$; we set
$r=\sqrt{A/\log x}$. 
The statements and proofs of Lemmata \ref{lemma:Mellin-Perron-integral} and
\ref{lemma:gf-lower} carry over (only the range of $u$ has to be adapted). So we
have
\begin{multline}\label{eq:p-g-f-Omega}
  \sum_{\substack{n\in\calM(\calF)\\n\le x}}u^{\Omega_\calF(n)}\\
=\frac{1}{2\sqrt\pi}\exp\Bigl(2\sqrt{A}\sqrt{\log x}-\frac{2B(v)+1}{4}\log\log
  x+\frac{2B(v)-1}{4}\log
  A
  + C(v)  \Bigr)\\
  \times\Bigl(1+O\Bigl(\frac{1}{(\log x)^{1/10}}\Bigr)\Bigr)
\end{multline}
for $x\to\infty$, $u\in U(\sqrt{A/\log x})$ and $v=\log u\sqrt{\log x/A}$.

We now consider the moment generating function
\begin{equation*}
  \E(e^{\Omega_\calF(N)t}) = \frac{\sum_{\substack{n\in\calM(\calF)\\n\le
        x}}e^{\Omega_\calF(n)t}}{\sum_{\substack{n\in\calM(\calF)\\n\le
        x}}1}.
\end{equation*}
Equation~\eqref{eq:p-g-f-Omega} yields
\begin{multline*}
  \E(e^{\Omega_\calF(N)t}) = \exp\Bigl(\frac12 (\log \log x - \log A) (B(0)-B(v)) + C(v)-C(0)\Bigr)\\\times\Bigl(1+O\Bigl(\frac{1}{(\log x)^{1/10}}\Bigr)\Bigr),
\end{multline*}
where $v = t \sqrt{\log x}/\sqrt{A}$, for all $t$ such that $|t| \leq v_0\sqrt{A}/(2\sqrt{\log x})$.
Since $B(v)=-\Lambda(0, v)$, Equation~\eqref{eq:Lambda-0-formula}  gives us
$B(0) - B(v) = v/\log \phi$. Likewise,
Lemmata~\ref{lemma:central-approximation-D} and
\ref{lemma:Dirichlet-series} yield
\begin{equation*}
C(v) - C(0) = - \sum_{m\in \calF} \Big( \log \Big( 1 - \frac{v}{\log m} \Big) + \frac{v}{\log m} \Big) + \kappa_1v,
\end{equation*}
so finally
\begin{multline*}
  \E(e^{\Omega_\calF(N)t}) = \\
\exp\Bigl(t \Big(\frac{a_1}{2} \sqrt{\log x}\log\log x + b_1 \sqrt{\log x} \Big) \Bigr) \prod_{m \in \calF} e^{-v/(\log m)} \Big( 1 - \frac{v}{\log m} \Big)^{-1}\\
\times  \Bigl(1+O\Bigl(\frac{1}{(\log x)^{1/10}}\Bigr)\Bigr),
\end{multline*}
where $b_1 = A^{-1/2}(\kappa_1 - (\log A)/(2\log \phi))$. Since $(1-v/\lambda)^{-1}$ is exactly the moment generating function of an $\operatorname{Exp}(\lambda)$-distributed random variable, Theorem~\ref{theorem:distribution-Omega} follows immediately from Curtiss's theorem in the same way as Theorem~\ref{theorem:asymptotic-normality}.
\end{proof}

\bibliography{bib/cheub}

\providecommand{\Submitted}{Submitted} \providecommand{\availableat}{ available
  at } \providecommand{\alsoavailableat}{ also available at }
  \providecommand{\evavailableat}{earlier version available at }
  \providecommand{\toappearin}{To appear in } \providecommand{\toappear}{to
  appear} \providecommand{\inpreparation}{in preparation}
  \providecommand{\doi}[1]{\href{http://dx.doi.org/#1}{\path{doi:#1}}}
  \providecommand{\lowercaseforams}{}
  \providecommand{\etc}{\emph{etc.}}\def\cprime{$'$}
\providecommand{\bysame}{\leavevmode\hbox to3em{\hrulefill}\thinspace}
\providecommand{\MR}{\relax\ifhmode\unskip\space\fi MR }
\providecommand{\MRhref}[2]{%
  \href{http://www.ams.org/mathscinet-getitem?mr=#1}{#2}
}
\providecommand{\href}[2]{#2}
\begin{thebibliography}{10}

\bibitem{Apostol:1976:introduction}
Tom~M. Apostol, \emph{Introduction to analytic number theory}, Springer-Verlag,
  New York-Heidelberg, 1976, Undergraduate Texts in Mathematics. \MR{0434929}

\bibitem{Bilu-Hanrot-Voutier:2001:exist-lucas-lehmer}
Yuri Bilu, Guillaume Hanrot, and Paul~M. Voutier,
  \href{http://dx.doi.org/10.1515/crll.2001.080}{\emph{Existence of primitive
  divisors of {L}ucas and {L}ehmer numbers}}, J. Reine Angew. Math.
  \textbf{539} (2001), 75--122, \lowercaseforams With an appendix by M.
  Mignotte. \MR{MR1863855 (2002j:11027)}

\bibitem{Carmichael:1913:numer-factor}
Robert~D. Carmichael, \href{http://dx.doi.org/10.2307/1967797}{\emph{On the
  numerical factors of the arithmetic forms $\alpha^n\pm \beta^n$}}, Ann. Math.
  \textbf{15} (1913), 30--70.

\bibitem{Curtiss:1942}
John~H. Curtiss, \href{http://dx.doi.org/10.1214/aoms/1177731541}{\emph{A note
  on the theory of moment generating functions}}, Ann. Math. Statistics
  \textbf{13} (1942), 430--433. \MR{0007577}

\bibitem{NIST:DLMF:v1.0.10}
\href{http://dlmf.nist.gov/}{\emph{{NIST} {D}igital library of mathematical
  functions}}, \url{http://dlmf.nist.gov/}, Release 1.0.10 of 2015-08-07, 2015,
  \lowercaseforams Online companion to \cite{Olver:2010:NHMF}.

\bibitem{Elliott:1980:probabilistic}
Peter D. T.~A. Elliott,
  \href{http://dx.doi.org/10.1007/978-1-4612-9992-9}{\emph{Probabilistic number
  theory. {II}. {C}entral limit theorems}}, Grundlehren der Mathematischen
  Wissenschaften, vol. 240, Springer-Verlag, Berlin-New York, 1980. \MR{560507}

\bibitem{Erdoes-Kac:1940:gauss}
Paul Erd{\H o}s and Mark Kac,
  \href{http://dx.doi.org/10.2307/2371483}{\emph{The {G}aussian law of errors
  in the theory of additive number theoretic functions}}, Amer. J. Math.
  \textbf{62} (1940), 738--742. \MR{0002374}

\bibitem{Erdos-Lehner:1941:distribution}
Paul Erd{\H o}s and Joseph Lehner,
  \href{http://dx.doi.org/10.1215/S0012-7094-41-00826-8}{\emph{The distribution
  of the number of summands in the partitions of a positive integer}}, Duke
  Math. J. \textbf{8} (1941), 335--345. \MR{0004841}

\bibitem{Flajolet-Gourdon-Dumas:1995:mellin}
Philippe Flajolet, Xavier Gourdon, and Philippe Dumas,
  \href{http://dx.doi.org/10.1016/0304-3975(95)00002-E}{\emph{Mellin transforms
  and asymptotics: {H}armonic sums}}, Theoret. Comput. Sci. \textbf{144}
  (1995), 3--58. \MR{96h:68093}

\bibitem{Flajolet-Grabner-Kirschenhofer-Prodinger:1994:mellin}
Philippe Flajolet, Peter Grabner, Peter Kirschenhofer, Helmut Prodinger, and
  Robert~F. Tichy,
  \href{http://dx.doi.org/10.1016/0304-3975(92)00065-Y}{\emph{Mellin transforms
  and asymptotics: digital sums}}, Theoret. Comput. Sci. \textbf{123} (1994),
  291--314. \MR{94m:11090}

\bibitem{Flajolet-Prodinger:1986:regis}
Philippe Flajolet and Helmut Prodinger,
  \href{http://dx.doi.org/10.1137/0215046}{\emph{Register allocation for
  unary-binary trees}}, SIAM J. Comput. \textbf{15} (1986), no.~3, 629--640.
  \MR{850413 (87j:68052)}

\bibitem{Goh-Schmutz:1995:number}
William M.~Y. Goh and Eric Schmutz,
  \href{http://dx.doi.org/10.1016/0097-3165(95)90111-6}{\emph{The number of
  distinct part sizes in a random integer partition}}, J. Combin. Theory Ser. A
  \textbf{69} (1995), no.~1, 149--158. \MR{MR1309156 (95k:11132)}

\bibitem{Knopfmacher:1975}
John Knopfmacher, \emph{Abstract analytic number theory}, North-Holland
  Publishing Co., Amsterdam-Oxford; American Elsevier Publishing Co., Inc., New
  York, 1975. \MR{0419383}

\bibitem{Knopfmacher-Zhang:2001}
John Knopfmacher and Wen-Bin Zhang,
  \href{http://dx.doi.org/10.1201/9780203908150}{\emph{Number theory arising
  from finite fields}}, Monographs and Textbooks in Pure and Applied
  Mathematics, vol. 241, Marcel Dekker, Inc., New York, 2001. \MR{1835434}

\bibitem{Luca-Pomerance-Wagner:2011:fibon}
Florian Luca, Carl Pomerance, and Stephan Wagner,
  \href{http://dx.doi.org/10.1016/j.jnt.2010.09.010}{\emph{Fibonacci
  integers}}, J. Number Theory \textbf{131} (2011), no.~3, 440--457.
  \MR{2739045}

\bibitem{Luca-Porubsky:2003:multiplicative}
Florian Luca and {\v{S}}tefan Porubsk{\'y},
  \href{http://www.fq.math.ca/Scanned/41-2/luca.pdf}{\emph{The multiplicative
  group generated by the {L}ehmer numbers}}, Fibonacci Quart. \textbf{41}
  (2003), no.~2, 122--132. \MR{1990520}

\bibitem{OEIS:2016}
\emph{The {O}n-{L}ine {E}ncyclopedia of {I}nteger {S}equences},
  \url{http://oeis.org}, 2016.

\bibitem{Olver:2010:NHMF}
Frank W.~J. Olver, Daniel~W. Lozier, Ronald~F. Boisvert, and Charles~W. Clark
  (eds.),
  \href{http://www.cambridge.org/us/academic/subjects/mathematics/abstract-analysis/nist-handbook-mathematical-functions}{\emph{{NIST}
  {H}andbook of mathematical functions}}, Cambridge University Press, New York,
  2010.

\bibitem{Wehmeier:2007}
S.~Wehmeier, \href{http://dx.doi.org/10.1007/s10986-007-0024-8}{\emph{The
  {E}rd{\H o}s-{K}ac theorem for additive arithmetical semigroups}}, Lith.
  Math. J. \textbf{47} (2007), no.~3, 352--360.

\bibitem{Whittaker-Watson:1996}
Edmund~T. Whittaker and George~N. Watson, \emph{A course of modern analysis},
  Cambridge University Press, Cambridge, 1996, \lowercaseforams Reprint of the
  fourth (1927) edition.

\end{thebibliography}
\bibliographystyle{amsplainurl}
\end{document}